\pdfoutput=1
\documentclass[12pt,reqno]{amsart}
\usepackage[margin=1in]{geometry}
\usepackage{amsmath}
\usepackage{amssymb}
\usepackage{graphicx}
\usepackage{booktabs}
\usepackage{array}
\usepackage{float}
\usepackage{caption}
\usepackage{enumitem}
\usepackage{amsthm}
\usepackage{hyperref}
\usepackage{longtable}
\usepackage{makecell}
\usepackage{fancyhdr}
\fancypagestyle{plain}{
  \fancyhf{}
  \fancyhead[L]{Dual Affine Spiral Orbits on $\mathbb{Z}^2$}
  \fancyhead[R]{C. Wrathall}
  \fancyfoot[C]{\thepage}
  
}
\title{Dual Affine Spiral Orbits on $\mathbb{Z}^2$ Generated by Paired Unit Squares}
\author{Chuck Wrathall \\
        Independent Researcher \\
        \href{https://orcid.org/0009-0004-7008-4489}{0009-0004-7008-4489}}
\theoremstyle{plain}
\newtheorem{theorem}{Theorem}[section]
\theoremstyle{remark}
\newtheorem*{remark}{Remark}
\begin{document}
\maketitle
\begin{abstract}
We introduce a simple iterative geometric construction on the integer lattice $\mathbb{Z}^2$ consisting of paired unit squares that share a single corner. At each step, a new square is constructed outwardly on the hypotenuse of the isosceles right triangle formed by the outer edges adjacent to the shared corner. This process generates two interlocking affine spiral orbits on $\mathbb{Z}^2$: a positive orbit $\{P_n\}$ starting at $(0, 0)$ and a negative orbit $\{N_n\}$ starting at $(-1, 2)$.
Both sequences satisfy linear recurrences driven by multiplication by the Gaussian integer $1 + i$. We show that the paired points remain symmetric with respect to the fixed midpoint $M = (-\frac{1}{2}, 1)$ for all $n$, satisfying the invariant $P_n + N_n = (-1, 2)$. Extending the iteration backward under the associated inverse maps produces an iterated function system whose attractor is similar to the Twindragon fractal, providing a concrete lattice-based viewpoint on its geometry. In addition, the paired points yield the normalized integer sequence
\[
a(n) = \frac{|P_n|^2 + |N_n|^2}{5} = 2^{n+1} + 1 - 2 \operatorname{Re}((1 + i)^n),
\]
which is always an integer and appears as A396151 in the OEIS.
\end{abstract}
\section{Introduction}
While experimenting with paired unit squares on graph paper, the author observed that repeatedly constructing new squares outwardly on the hypotenuse of the isosceles right triangle formed at each step produced two interlocking spiral sequences of lattice points. At every iteration, the two unit squares are positioned as point reflections of each other through their shared corner. These sequences remain on $\mathbb{Z}^2$ indefinitely, satisfy clean affine recurrences, and are related by a simple point-reflection symmetry through the non-lattice point $M = (-\frac{1}{2}, 1)$. When the iteration is extended backward ($n < 0$), the inverse maps turn out to be closely related to the two contractions that generate the famous Twindragon fractal. As $n \to -\infty$, $P_n \to N_0$ and $N_n \to P_0$: the two seed points swap roles as limit points, and the orbits spiral around $M$ with their distance to it approaching $\sqrt{5}/2$. Thus, an elementary geometric construction on the integer lattice leads naturally to an iterated function system whose attractor is similar to the classical Twindragon fractal. This note presents the construction in full detail, derives the affine recurrences and their closed forms in the complex plane, proves the main invariants (including the new integer sequence $a(n)$), and clarifies the connection to the Twindragon attractor.
\section{Geometric Construction}
\subsection{The Lattice}
The construction is built entirely on the integer lattice $\mathbb{Z}^2$. Any point in the lattice can be chosen as the initial corner $P_0 = (0, 0)$; this choice is a labeling convention enabled by the translational symmetry of the lattice.
\subsection{Base Iteration}
At the base iteration $n = 0$, the construction consists of two unit squares that share exactly one common corner, denoted $P_0 = (0,0)$. These two squares are positioned as point reflections of each other through the shared corner. This configuration establishes the initial paired-square structure from which all subsequent iterations are generated. The corner coordinates of each square at the base iteration are listed below, and correspond directly to the $n = 0$ entries of Appendix~A.
\begin{table}[H]
\centering
\begin{tabular}{llc}
\toprule
\textbf{Square} & \textbf{Corners} & \textbf{Shared Corner} \\
\midrule
$a$ & $(0,0),(-1,0),(0,1),(-1,1)$ & $P_0 = (0,0)$ \\
$b$ & $(0,0),(1,0),(0,-1),(1,-1)$ & $P_0 = (0,0)$ \\
$-a$ & $(-1,2),(0,2),(-1,1),(0,1)$ & $N_0 = (-1,2)$ \\
$-b$ & $(-1,2),(-2,2),(-1,3),(-2,3)$ & $N_0 = (-1,2)$ \\
\bottomrule
\end{tabular}
\caption{Corner coordinates of the base paired squares for both orbits.}
\end{table}
\subsection{Corner Sum Identity}
At each iteration $n \ge 0$, a pair of unit squares shares a common corner, denoted $P_n$ (positive orbit) or $N_n$ (negative orbit). This identity was the original observation that identified the shared corners as geometrically meaningful iteration points: the average of the eight corner coordinates of each pair of squares always returns the shared corner itself,
\[
P_n = \frac{1}{8} \sum_{i=1}^{8} (x_i, y_i), \qquad N_n = \frac{1}{8} \sum_{i=1}^{8} (x_i, y_i).
\]
In each case, because the two unit squares are point reflections of each other through their shared corner, the eight corner coordinates are symmetric with respect to that corner. Therefore the shared corner is exactly the centroid (average) of those eight points. This identity holds for every iteration of both orbits, and is verified explicitly for $n = 0$ through $8$ in Appendix~A.
\subsection{Hypotenuse Line Construction}
The two outer edges adjacent to the shared corner form the legs of an isosceles right triangle with the right angle at the shared corner. The hypotenuse construction at step $n$ is the line containing the hypotenuse of this triangle. A new square is constructed outwardly on this hypotenuse line, and its far corner becomes the next iteration point $P_{n+1}$ (positive orbit) or $N_{n+1}$ (negative orbit).
\begin{figure}[H]
\centering
\begin{minipage}{0.4\textwidth}
\centering
\includegraphics[width=\textwidth]{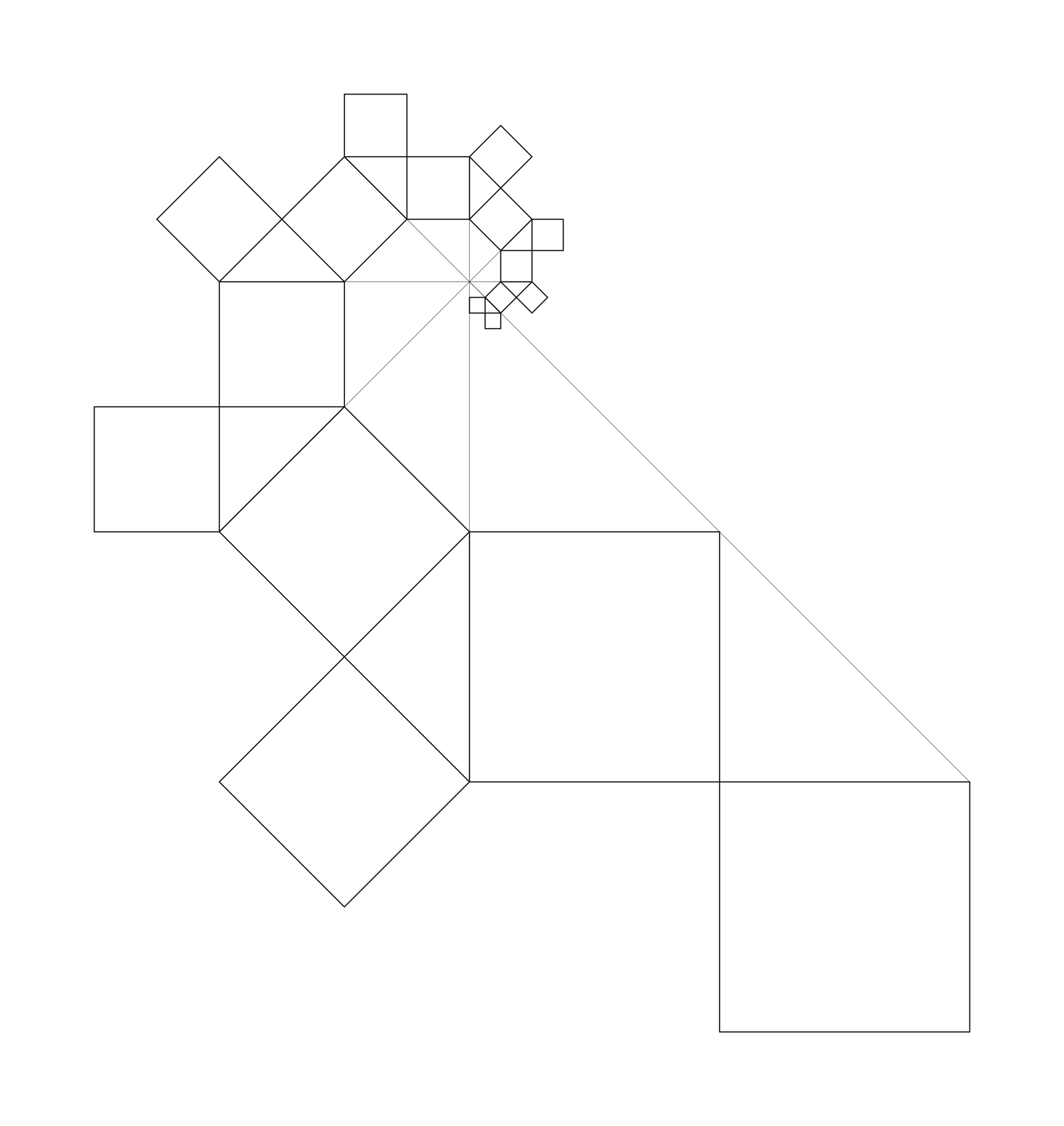}
\caption{Positive orbit with hypotenuse line construction.}
\end{minipage}
\hfill
\begin{minipage}{0.4\textwidth}
\centering
\includegraphics[width=\textwidth]{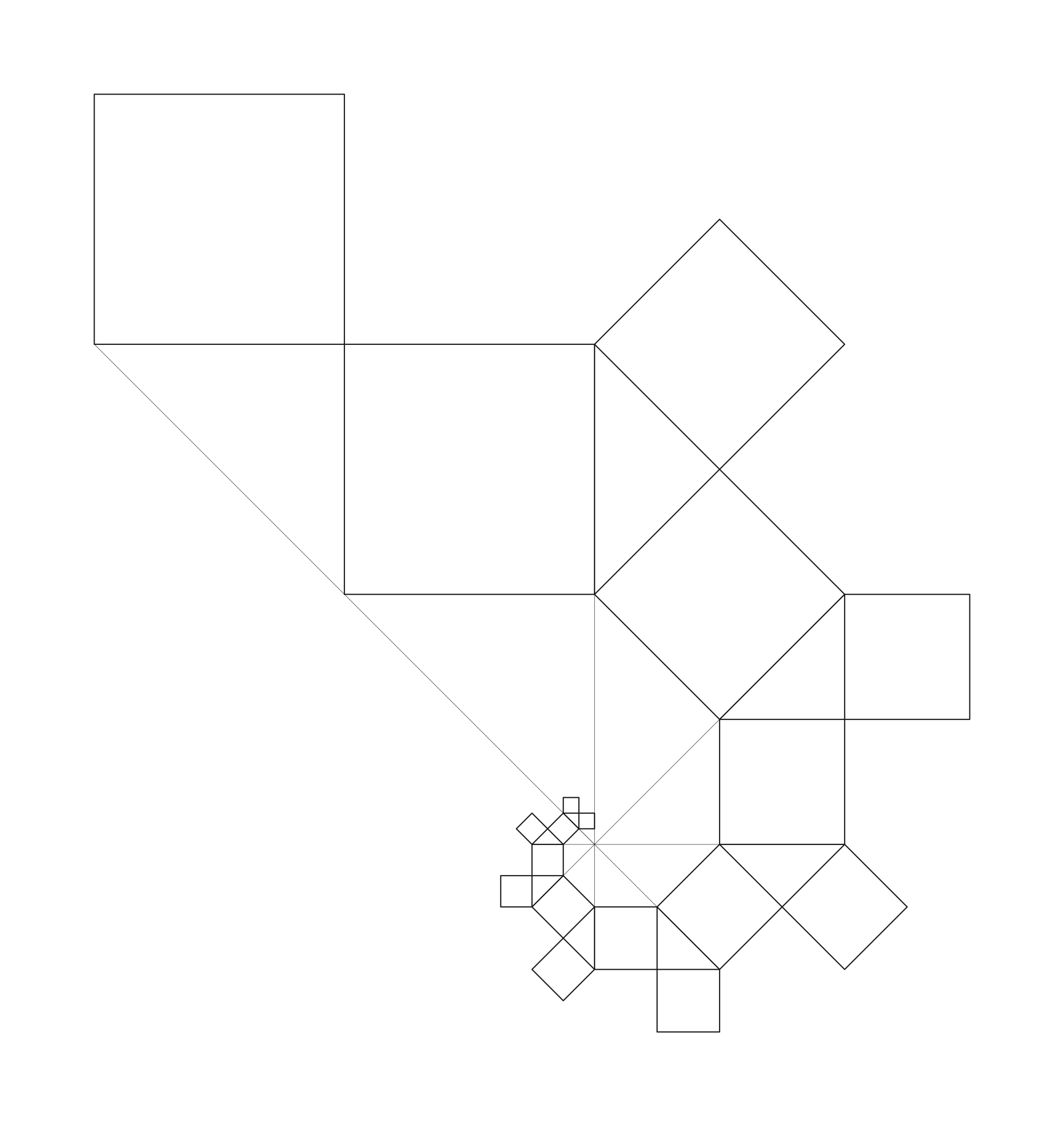}
\caption{Negative orbit with hypotenuse line construction.}
\end{minipage}
\end{figure}
The hypotenuse lines generated by the positive orbit are all concurrent at the point $N_0 = (-1,2)$, and this common intersection point naturally defines the starting point of the negative orbit. Symmetrically, all hypotenuse lines of the negative orbit pass through $P_0 = (0,0)$. Within each orbit the lines take only the four directions produced by repeated $45^\circ$ rotation (horizontal, vertical, and the two diagonals), so each orbit contributes exactly four distinct lines through its common point. Pairing each positive-orbit line with the negative-orbit line of the same direction gives four pairs of parallel lines, with perpendicular separations $\frac{1}{\sqrt{2}}$, 1, 2, and $\frac{3}{\sqrt{2}}$.
\begin{figure}[H]
\centering
\includegraphics[width=0.4\textwidth]{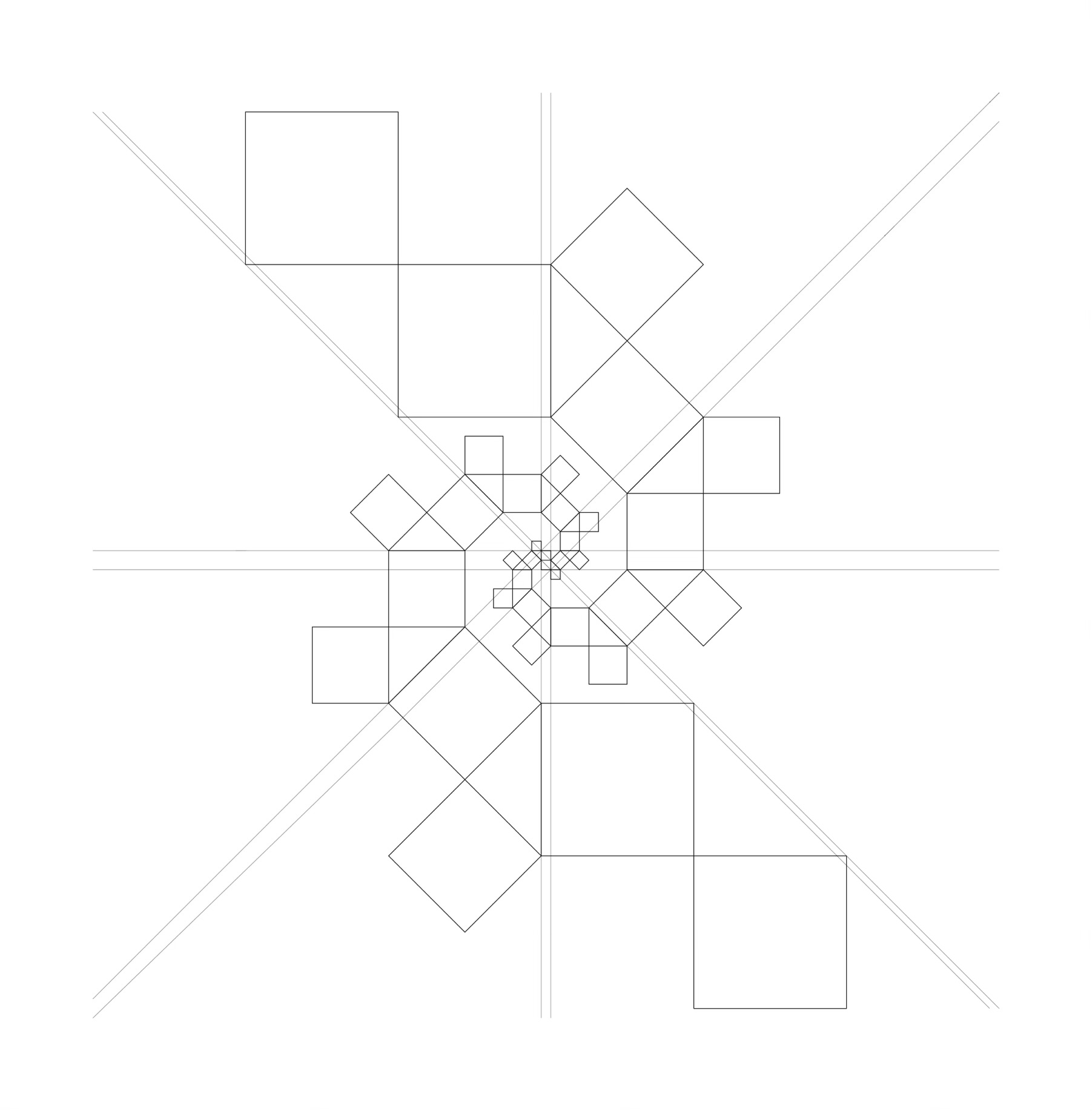}
\caption{Positive and negative orbits on $\mathbb{Z}^2$ and the four families of parallel hypotenuse lines.}
\end{figure}
To connect the geometric construction to the algebra, consider the base case at \(P_0 = (0,0)\). The paired unit squares have outer edges from \((0,0)\) to \((1,0)\) and from \((0,0)\) to \((0,1)\) that form an isosceles right triangle whose hypotenuse runs from \((0,1)\) to \((1,0)\), with direction vector \((1,-1)\). Rotating this vector \(90^\circ\) counterclockwise produces the outward perpendicular \((1,1)\). Adding it to the hypotenuse endpoints yields the new corners \((1,2)\) and \((2,1)\). The construction selects $P_1 = (2,1)$ as the next iteration point of the positive orbit.
This single step is precisely the affine transformation that will be written below using the matrix $A$ and translation vector $v$. The same local rule generates the full positive orbit on $\mathbb{Z}^2$.
Each step scales side lengths by $\sqrt{2}$ and rotates vectors counterclockwise by $\pi/4$. The construction therefore has 8-step rotational periodicity: after eight iterations the orientation repeats while all displacements scale by $16$: the positive orbit about $N_0$ and the negative orbit about $P_0$, each about the other orbit's seed point.
This iterative process generates two sequences of iteration points on $\mathbb{Z}^2$: the positive orbit $\{P_n\}$ starting at $(0,0)$ and the negative orbit $\{N_n\}$ starting at $(-1,2)$.
\subsection{Affine Recurrences}
The first nine iteration points of the positive orbit are: $(0,0)_0$, $(2,1)_1$, $(3,4)_2$, $(1,8)_3$, $(-5,10)_4$, $(-13,6)_5$, $(-17,-6)_6$, $(-9,-22)_7$, $(15,-30)_8$, \dots
The first nine iteration points of the negative orbit are: $(-1,2)_0$, $(-3,1)_1$, $(-4,-2)_2$, $(-2,-6)_3$, $(4,-8)_4$, $(12,-4)_5$, $(16,8)_6$, $(8,24)_7$, $(-16,32)_8$, \dots
Both sequences satisfy linear recurrences on $\mathbb{Z}^2$. Define the matrix
\[
A = \begin{pmatrix} 1 & -1 \\ 1 & 1 \end{pmatrix}
\]
and the vector
\[
v = (2,1).
\]
The positive orbit satisfies the affine recurrence
\[
P_{n+1} = A P_n + v, \quad \text{with } P_0 = (0,0).
\]
\begin{figure}[H]
\centering
\includegraphics[width=0.9\textwidth]{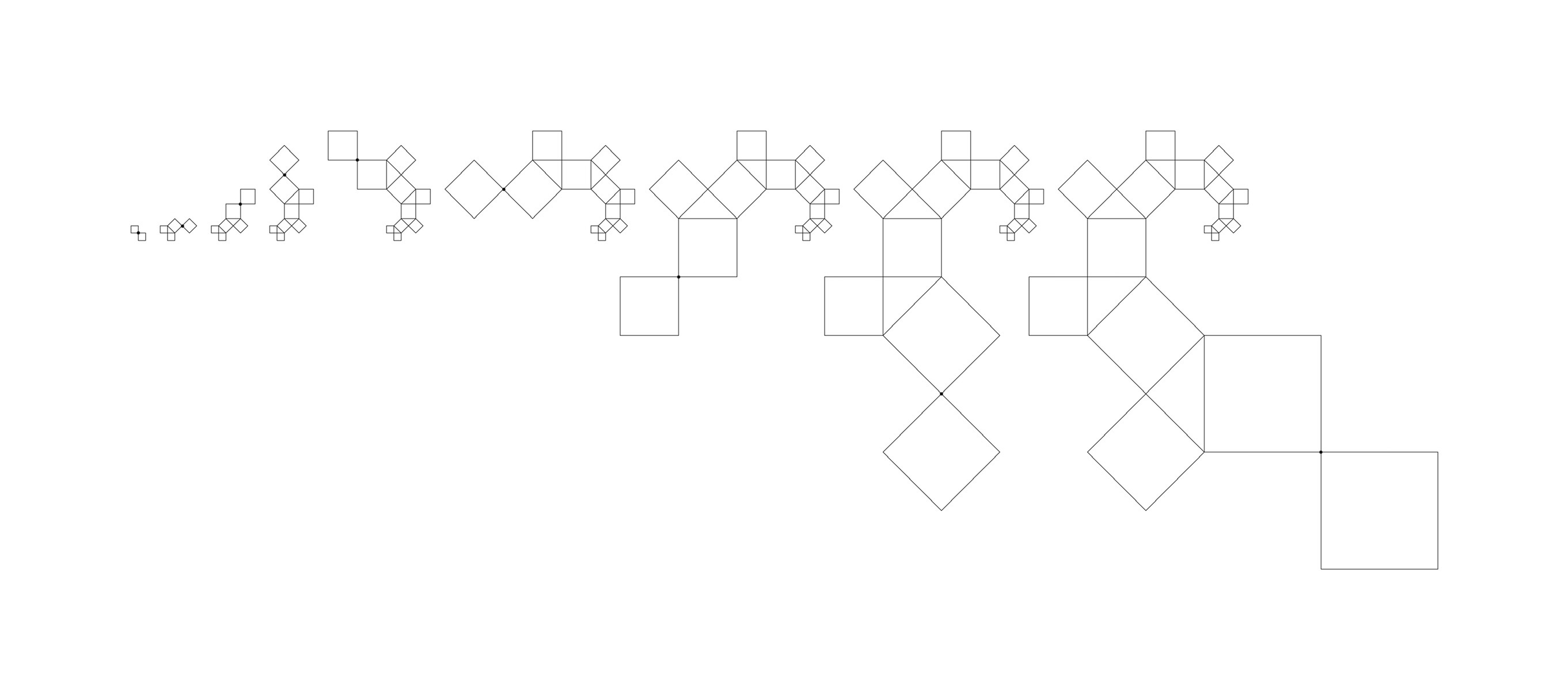}
\caption{Iterations $n = 0$ to 8 of the positive orbit using paired squares.}
\end{figure}
\begin{figure}[H]
\centering
\includegraphics[width=0.9\textwidth]{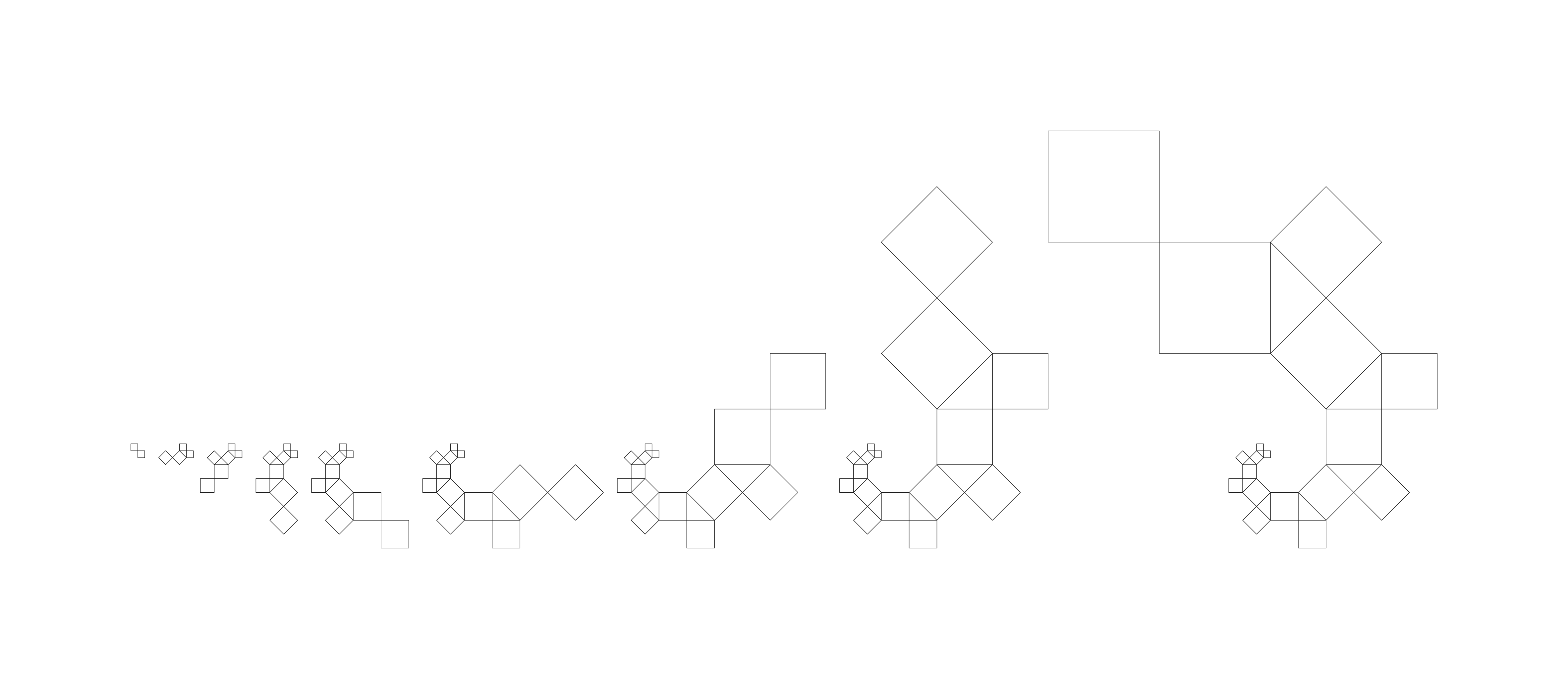}
\caption{Iterations $n = 0$ to 8 of the negative orbit using paired squares.}
\end{figure}
The negative orbit satisfies the linear recurrence
\[
N_{n+1} = A N_n, \quad \text{with } N_0 = (-1, 2).
\]
Identifying $\mathbb{R}^2$ with the complex plane $\mathbb{C}$ via the map $(x, y) \mapsto x + iy$, the matrix $A$ corresponds to multiplication by the Gaussian integer $1+i$. Under this identification the recurrences become
\[
P_{n+1} = (1 + i)P_n + (2 + i), \quad \text{with } P_0 = 0,
\]
and
\[
N_n = (-1 + 2i) \cdot (1 + i)^n, \quad \text{with } N_0 = -1 + 2i.
\]
Solving these recurrences yields the closed forms
\[
P_n = (-1 + 2i)\bigl[1 - (1 + i)^n\bigr], \quad N_n = (-1 + 2i) \cdot (1 + i)^n.
\]
Multiplication by $1 + i$ rotates vectors counterclockwise by $\pi/4$ and scales their lengths by $\sqrt{2}$. In particular, the negative-orbit points $N_n = (-1 + 2i)(1 + i)^n$ are exact integer samples of a logarithmic spiral about $P_0$, the radius scaling by $\sqrt{2}$ for each $\pi/4$ of rotation; correspondingly $P_n = N_0 - N_n$ samples a logarithmic spiral about $N_0$. This is the precise sense in which the orbits are ``spiral.''
\subsection{Unified Geometric Encoding}
The single Gaussian integer $(1+i)^n$ encodes all the geometric information of the $n$-th iteration simultaneously. Writing $(1+i)^n = a_n + b_n i$ with $a_n, b_n \in \mathbb{Z}$, the orbit points are
\[
N_n = (-1+2i)(a_n + b_n i) = (-a_n - 2b_n) + (2a_n - b_n)i,
\]
and $P_n = (-1+2i) - N_n$, recovering both lattice points directly. The squared side length of the construction at step $n$ is $|(1+i)^n|^2 = 2^n$, so each side has length $2^{n/2}$, and the corresponding square has area $2^n$. The argument $\arg(1+i)^n = n\pi/4$ gives the cumulative rotation angle, so the orientation of the paired squares at step $n$ is $n \times 45^\circ$ from the base configuration. All four quantities (the coordinates of $P_n$ and $N_n$, the side length $2^{n/2}$, the area $2^n$, and the orientation angle $n\pi/4$) are thus read off from the single complex number $(1+i)^n$.
\section{Fixed Midpoint M}
The positive and negative orbits are symmetric with respect to a fixed point in the plane. For all $n \ge 0$ the iteration points satisfy the invariant relation
\[
P_n + N_n = (-1, 2).
\]
Dividing by 2 gives the midpoint
\[
M = \bigl(-\frac{1}{2}, 1\bigr).
\]
This point is independent of $n$. Consequently, $N_n$ is the point reflection of $P_n$ through $M$.
$M$ is the only point in the construction that does not lie on the integer lattice $\mathbb{Z}^2$. All iteration points $P_n$ and $N_n$, as well as all corners of the squares, lie on $\mathbb{Z}^2$.
The points $P_0 = (0, 0)$, $M = (-\frac{1}{2}, 1)$, and $N_0 = (-1, 2)$ are collinear along the direction $(-1, 2)$, and the distances satisfy $|P_0M| = |MN_0| = \sqrt{5}/2$.
Extending the recurrence backward ($n < 0$) does not collapse the orbits to $M$. Instead, as $n \to -\infty$, $P_n \to N_0$ and $N_n \to P_0$, so the points spiral around $M$ with the common distance $d_n = |P_n - M|$ approaching $\sqrt{5}/2$. This backward behaviour is the contracting dynamics whose attractor is similar to the Twindragon fractal.
\subsection{Distance from the Fixed Midpoint}
The positive and negative orbits remain symmetric with respect to the fixed point
\[
M = \left(-\frac12, 1\right)
\]
for all $n$. Define the common distance
\[
d_n = |P_n - M| = |N_n - M|.
\]
The midpoint relation $P_n + N_n = 2M$ immediately implies
\[
P_n - M = -(N_n - M).
\]
Applying the parallelogram law then gives
\[
|P_n|^2 + |N_n|^2 = 2d_n^2 + 2|M|^2.
\]
Since $|M|^2 = \frac54$, this simplifies to
\[
|P_n|^2 + |N_n|^2 = 2d_n^2 + \frac52.
\]
Dividing by 5 and using the definition of the sequence yields
\[
a(n) = \frac25\, d_n^2 + \frac12,
\]
or equivalently
\[
d_n^2 = \frac52 \left( a(n) - \frac12 \right).
\]
Thus $a(n)$ admits a direct geometric meaning: it is an affine rescaling of the squared distance from the paired points to the symmetry center $M$.
\section{Connection to the Twindragon Fractal}
The forward iteration of the two orbits is governed by the affine maps
\[
T(z) = (1 + i)z + (2 + i), \qquad S(z) = (1 + i)z,
\]
with $P_0 = 0$ and $N_0 = -1 + 2i$. The corresponding inverse maps are
\[
T^{-1}(z) = \frac{z - (2 + i)}{1 + i}, \qquad S^{-1}(z) = \frac{z}{1 + i}.
\]
Both inverse maps are contractions: each multiplies distances by $1/\sqrt{2}$ and rotates by $-\pi/4$. The pair $\{T^{-1}, S^{-1}\}$ therefore forms an iterated function system (IFS) in the sense of Barnsley~\cite{barnsley}.
The standard Twindragon fractal is the attractor of the IFS
\[
f_1(z) = \frac{z}{1+i}, \qquad f_2(z) = 1 - \frac{z}{1+i},
\]
whose two maps are similarities of ratio $1/\sqrt{2}$, and indeed $S^{-1} = f_1$ exactly.
Our generators $S^{-1}$ and $T^{-1}$ share the common linear part $z \mapsto z/(1+i)$, so each
is a similarity of ratio $1/\sqrt{2}$ and rotation $-\pi/4$, the two differing only by the
translation
\[
T^{-1}(z) - S^{-1}(z) = -\frac{2+i}{1+i} = \frac{-3+i}{2}.
\]
An iterated function system whose maps share a single contractive linear part and differ only by a translation has a self-affine set as its attractor. Here the common linear part is a similarity and the digit set $\{0,1\}$ is a complete residue system modulo $1+i$, so this attractor is a self-similar tile of positive area~\cite{bandt, grochenig, lagariaswang}. Writing $\varepsilon_k \in \{0,1\}$ for the choice between $S^{-1}$ and $T^{-1}$ at step $k$ (with $\varepsilon_k = 1$ recording an application of $T^{-1}$), this attractor is
\[
\mathcal{A} = \left\{ -\frac{2+i}{1+i}\sum_{k=0}^{\infty}\varepsilon_k (1+i)^{-k}
   : \varepsilon_k \in \{0,1\} \right\} = (-2-i)\,\mathcal{T},
\]
where
\[
\mathcal{T} = \left\{ \sum_{k=1}^{\infty} \delta_k (1+i)^{-k} : \delta_k \in \{0,1\} \right\}
\]
is the base-$(1+i)$, digit-$\{0,1\}$ number-system tile \cite{knuth, gilbert}. (The base $1+i$ used here is the mirror-image convention of Knuth's base $i-1$; the two tiles are reflections of one another, hence the same fractal up to a reflection.) This tile is
the Twindragon: scaling it by $1+i$ yields two adjacent copies of itself, the defining
rep-tile property of the Twindragon. Hence $\mathcal{A}$ is the image of the Twindragon under
the similarity $z \mapsto (-2-i)z$, and is in particular similar to it
\cite[Ch.~3]{barnsley}. The Twindragon is the smallest of the self-affine lattice tiles arising from expanding Gaussian-integer bases, the admissible bases being exactly the canonical number systems classified by K\'atai and Szab\'o~\cite{katai}.
The fixed point of $S^{-1}$ is $P_0 = 0$ and the fixed point of $T^{-1}$ is $N_0 = -1 + 2i$. The midpoint $M = (-\frac12, 1)$ established in Section~3 is the center of this attractor. It is the unique point fixed by the reflection exchanging the two orbits, and it satisfies the IFS centroid identity
\[
M = \tfrac12\bigl(S^{-1}(M) + T^{-1}(M)\bigr).
\]
When the recurrence is extended to negative indices, the orbits do not collapse to $M$. Instead, as $n \to -\infty$, $P_n \to N_0$ and $N_n \to P_0$, with the common distance $d_n = |P_n - M|$ approaching $\sqrt{5}/2$ rather than zero. The backward orbits trace deterministic paths through this similar copy of the Twindragon; the whole tile is recovered only as the closure of the points obtained from all $\{0,1\}$-sequences of $S^{-1}$ and $T^{-1}$, not from any single itinerary. The 8-step rotational symmetry follows from $(1+i)^8 = 16$: after eight backward iterations, deviations from the fixed points scale by a factor of $1/16$.
\begin{figure}[H]
\centering
\includegraphics[width=0.9\textwidth]{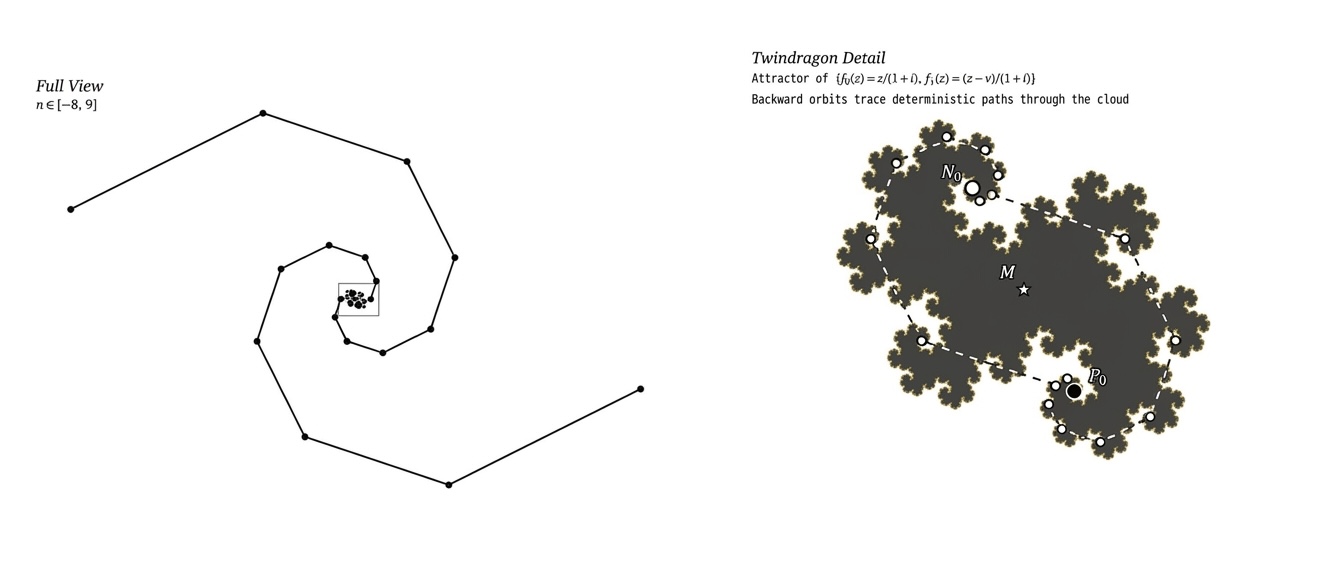}
\caption{Bidirectional view: forward orbits ($n \ge 0$) expand outward on the integer lattice $\mathbb{Z}^2$, while backward orbits ($n < 0$) generate points within a similar image of the Twindragon attractor.}
\end{figure}
\begin{figure}[H]
\centering
\includegraphics[width=0.8\textwidth]{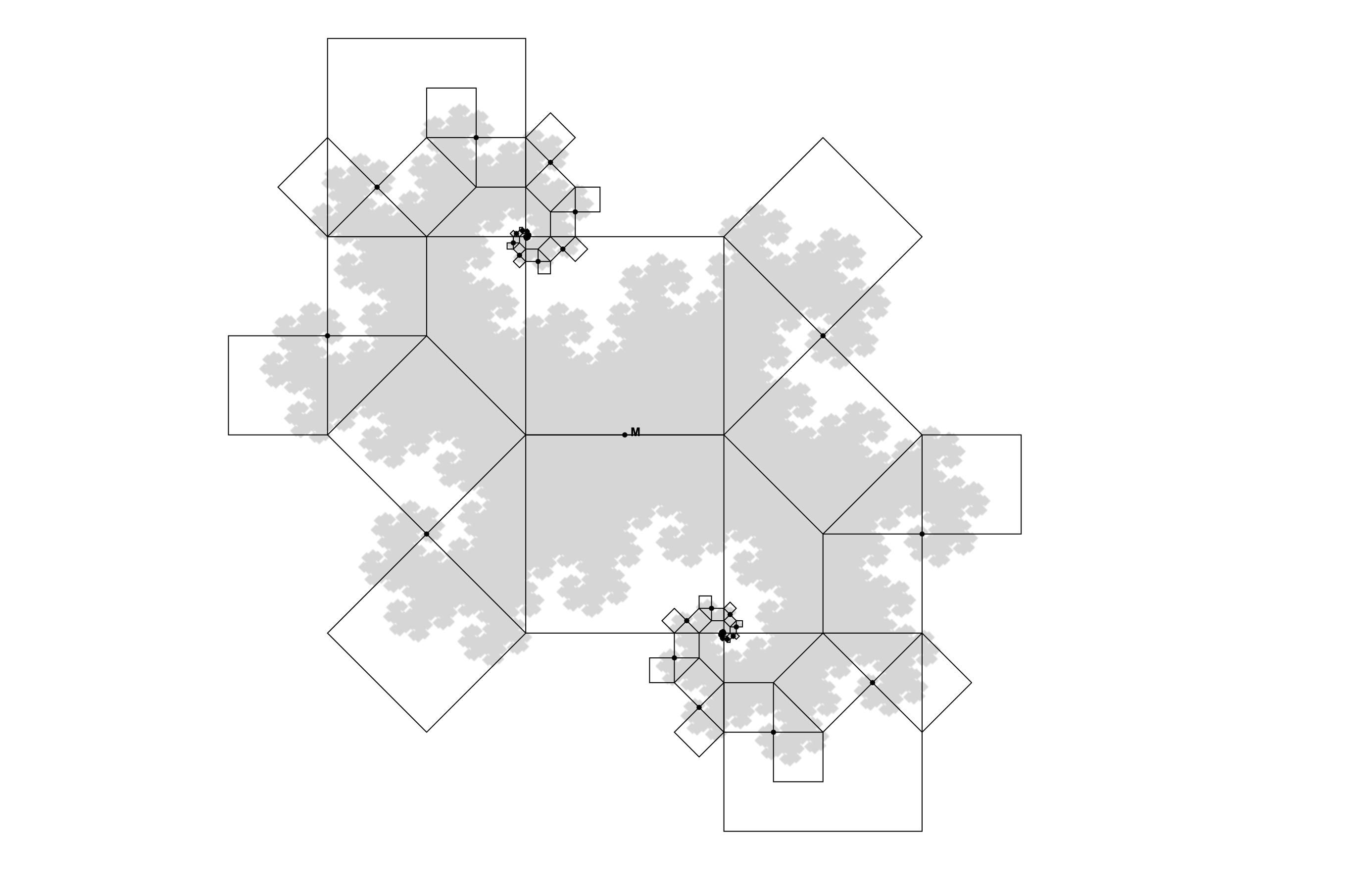}
\caption{Backward iterations of the paired squares. As $n$ decreases, the configurations contract and spiral into the limit points $N_0$ and $P_0$ (the fixed points of $T^{-1}$ and $S^{-1}$), accumulating on the shaded attractor of $\{T^{-1}, S^{-1}\}$, a similar image of the Twindragon. The point $M = \bigl(-\frac{1}{2}, 1\bigr)$ is its center of symmetry.}
\end{figure}
\section{Main Invariants}
\begin{theorem}[Sum Invariant]
For all $n \ge 0$,
\[
P_n + N_n = (-1, 2).
\]
\end{theorem}
\begin{proof}
Substituting the closed forms gives
\[
P_n + N_n = (-1 + 2i)[1 - (1 + i)^n] + (-1 + 2i)(1 + i)^n.
\]
The terms involving $(1 + i)^n$ cancel, leaving $-1 + 2i$.
\end{proof}
\begin{theorem}[Integrality of the Norm Sum]
Define
\[
a(n) = \frac{|P_n|^2 + |N_n|^2}{5}, \qquad n \ge 0,
\]
where $|U|^2 = x^2 + y^2$ denotes the squared Euclidean norm. Then $a(n)$ is an integer for every $n$, given in closed form by
\[
a(n) = 2^{n+1} + 1 - 2\operatorname{Re}\!\bigl((1 + i)^n\bigr).
\]
\end{theorem}
\begin{proof}
Using the closed forms
\[
P_n = (-1 + 2i)\bigl(1 - (1 + i)^n\bigr),
\]
and
\[
N_n = (-1 + 2i)(1 + i)^n,
\]
together with
\[
|-1 + 2i|^2 = 5,
\]
we obtain
\[
|P_n|^2 = 5\, |1 - (1 + i)^n|^2,
\]
and
\[
|N_n|^2 = 5\, |(1 + i)^n|^2.
\]
Expanding and simplifying yields
\[
a(n) = 2^{n+1} + 1 - 2\operatorname{Re}\!\bigl((1 + i)^n\bigr).
\]
Since $(1 + i)^n$ is a Gaussian integer for every $n$, its real part is an integer. Therefore $a(n)$ is always an integer. This sequence appears as A396151 in the OEIS~\cite{oeis}.
\end{proof}
\begin{remark}
The first terms of $a(n)$ are
\[
a(n) = 1,\ 3,\ 9,\ 21,\ 41,\ 73,\ 129,\ 241,\ 481,\ 993,\ 2049,\ 4161,\ \dots
\qquad (n = 0, 1, 2, \dots).
\]
Beyond the closed form above, $a(n)$ satisfies the integer linear recurrence
\[
a(n) = 5\,a(n-1) - 10\,a(n-2) + 10\,a(n-3) - 4\,a(n-4), \qquad n \ge 4,
\]
whose characteristic polynomial factors as $(x-1)(x-2)(x^2 - 2x + 2)$, with roots $1$, $2$, and $1 \pm i$. These four roots account for the three terms $1$, $2^{n+1}$, and $-2\operatorname{Re}((1+i)^n)$ of the closed form.
\end{remark}
\begin{theorem}[Distance Interpretation]
Let
\[
d_n = |P_n - M| = |N_n - M|.
\]
Then
\[
a(n) = \frac25\, d_n^2 + \frac12.
\]
Equivalently,
\[
d_n^2 = \frac52\left(a(n) - \frac12\right).
\]
\end{theorem}
\begin{proof}
This follows directly from the midpoint symmetry and the identity
\[
|P_n|^2 + |N_n|^2 = 2d_n^2 + \frac52.
\]
\end{proof}
\section{Conclusion}
This note introduces a simple geometric construction on the integer lattice $\mathbb{Z}^2$ using paired unit squares. The process generates two interlocking affine spiral orbits that remain on lattice points, satisfy clean recurrences driven by the Gaussian integer $1 + i$, and are symmetric through the fixed midpoint $M = (-\frac{1}{2}, 1)$. It also produces the new integer sequence $a(n)$ (OEIS A396151). Extending the iteration backward reveals that the same maps form an iterated function system whose attractor is similar to the Twindragon fractal. The pairing construction and the sequence $a(n)$ appear to be new, whereas the backward dynamics re-derives, from an elementary starting point, the classical self-affine tile of the Gaussian-integer base $1 + i$. The construction therefore offers a visual lattice viewpoint on the dynamics of this well known fractal, and the fact that such elementary geometry gives rise to questions touching algebraic number theory suggests it is richer than its origins on graph paper might imply.

\bigskip
\section*{Appendix A: Coordinate Tables for Paired Squares}
To verify the corner-sum identity, we list all eight corner coordinates of the paired unit squares for each iteration from $n = 0$ to $n = 8$, together with their sums and the average (total sum divided by 8) for both orbits. Since the two squares are point reflections of each other through their shared corner, that average is exactly the shared corner $P_n$ (or $N_n$).
\subsection*{Positive Orbit Paired Squares}
\footnotesize
\begin{longtable}{@{}p{2.2cm}p{8.0cm}p{3.0cm}p{1.6cm}@{}}
\toprule
\textbf{Iteration point} & \textbf{Pair of squares (corners)} & \textbf{Coordinate sums} & \textbf{Average (÷8)} \\
\midrule
\endfirsthead
\toprule
\textbf{Iteration point} & \textbf{Pair of squares (corners)} & \textbf{Coordinate sums} & \textbf{Average (÷8)} \\
\midrule
\endhead
\bottomrule
\endfoot
\(P_0 = (0,0)\) & \makecell[tl]{a: \((0,0),(-1,0),(0,1),(-1,1)\) \\ b: \((0,0),(1,0),(0,-1),(1,-1)\)} & \makecell[tl]{a: \((-2,2)\) \\ b: \((2,-2)\)} & \((0,0)\) \\[0.3em]
\(P_1 = (2,1)\) & \makecell[tl]{c: \((2,1),(1,0),(0,1),(1,2)\) \\ d: \((2,1),(3,0),(4,1),(3,2)\)} & \makecell[tl]{c: \((4,4)\) \\ d: \((12,4)\)} & \((2,1)\) \\[0.3em]
\(P_2 = (3,4)\) & \makecell[tl]{e: \((3,4),(3,2),(1,2),(1,4)\) \\ f: \((3,4),(5,4),(3,6),(5,6)\)} & \makecell[tl]{e: \((8,12)\) \\ f: \((16,20)\)} & \((3,4)\) \\[0.3em]
\(P_3 = (1,8)\) & \makecell[tl]{g: \((1,8),(1,4),(-1,6),(3,6)\) \\ h: \((1,8),(-1,10),(3,10),(1,12)\)} & \makecell[tl]{g: \((4,24)\) \\ h: \((4,40)\)} & \((1,8)\) \\[0.3em]
\(P_4 = (-5,10)\) & \makecell[tl]{i: \((-5,10),(-1,6),(-1,10),(-5,6)\) \\ j: \((-5,10),(-9,10),(-9,14),(-5,14)\)} & \makecell[tl]{i: \((-12,32)\) \\ j: \((-28,48)\)} & \((-5,10)\) \\[0.3em]
\(P_5 = (-13,6)\) & \makecell[tl]{k: \((-13,6),(-9,2),(-5,6),(-9,10)\) \\ l: \((-13,6),(-17,2),(-21,6),(-17,10)\)} & \makecell[tl]{k: \((-36,24)\) \\ l: \((-68,24)\)} & \((-13,6)\) \\[0.3em]
\(P_6 = (-17,-6)\) & \makecell[tl]{m: \((-17,-6),(-17,2),(-9,2),(-9,-6)\) \\ n: \((-17,-6),(-17,-14),(-25,-14),(-25,-6)\)} & \makecell[tl]{m: \((-52,-8)\) \\ n: \((-84,-40)\)} & \((-17,-6)\) \\[0.3em]
\(P_7 = (-9,-22)\) & \makecell[tl]{o: \((-9,-22),(-9,-6),(-1,-14),(-17,-14)\) \\ p: \((-9,-22),(-1,-30),(-9,-38),(-17,-30)\)} & \makecell[tl]{o: \((-36,-56)\) \\ p: \((-36,-120)\)} & \((-9,-22)\) \\[0.3em]
\(P_8 = (15,-30)\) & \makecell[tl]{q: \((15,-30),(15,-14),(-1,-14),(-1,-30)\) \\ r: \((15,-30),(31,-30),(31,-46),(15,-46)\)} & \makecell[tl]{q: \((28,-88)\) \\ r: \((92,-152)\)} & \((15,-30)\) \\
\end{longtable}
\subsection*{Negative Orbit Paired Squares}
\begin{longtable}{@{}p{2.2cm}p{8.0cm}p{3.0cm}p{1.6cm}@{}}
\toprule
\textbf{Iteration point} & \textbf{Pair of squares (corners)} & \textbf{Coordinate sums} & \textbf{Average (÷8)} \\
\midrule
\endfirsthead
\toprule
\textbf{Iteration point} & \textbf{Pair of squares (corners)} & \textbf{Coordinate sums} & \textbf{Average (÷8)} \\
\midrule
\endhead
\bottomrule
\endfoot
\(N_0 = (-1,2)\) & \makecell[tl]{-a: \((-1,2),(0,2),(-1,1),(0,1)\) \\ -b: \((-1,2),(-2,2),(-1,3),(-2,3)\)} & \makecell[tl]{-a: \((-2,6)\) \\ -b: \((-6,10)\)} & \((-1,2)\) \\[0.3em]
\(N_1 = (-3,1)\) & \makecell[tl]{-c: \((-3,1),(-2,2),(-1,1),(-2,0)\) \\ -d: \((-3,1),(-4,2),(-5,1),(-4,0)\)} & \makecell[tl]{-c: \((-8,4)\) \\ -d: \((-16,4)\)} & \((-3,1)\) \\[0.3em]
\(N_2 = (-4,-2)\) & \makecell[tl]{-e: \((-4,-2),(-4,0),(-2,0),(-2,-2)\) \\ -f: \((-4,-2),(-6,-2),(-4,-4),(-6,-4)\)} & \makecell[tl]{-e: \((-12,-4)\) \\ -f: \((-20,-12)\)} & \((-4,-2)\) \\[0.3em]
\(N_3 = (-2,-6)\) & \makecell[tl]{-g: \((-2,-6),(-2,-2),(0,-4),(-4,-4)\) \\ -h: \((-2,-6),(0,-8),(-4,-8),(-2,-10)\)} & \makecell[tl]{-g: \((-8,-16)\) \\ -h: \((-8,-32)\)} & \((-2,-6)\) \\[0.3em]
\(N_4 = (4,-8)\) & \makecell[tl]{-i: \((4,-8),(0,-4),(0,-8),(4,-4)\) \\ -j: \((4,-8),(8,-8),(8,-12),(4,-12)\)} & \makecell[tl]{-i: \((8,-24)\) \\ -j: \((24,-40)\)} & \((4,-8)\) \\[0.3em]
\(N_5 = (12,-4)\) & \makecell[tl]{-k: \((12,-4),(6,0),(4,-4),(6,-8)\) \\ -l: \((12,-4),(16,0),(20,-4),(16,-8)\)} & \makecell[tl]{-k: \((32,-16)\) \\ -l: \((64,-16)\)} & \((12,-4)\) \\[0.3em]
\(N_6 = (16,8)\) & \makecell[tl]{-m: \((16,8),(16,0),(8,0),(8,8)\) \\ -n: \((16,8),(16,16),(24,16),(24,8)\)} & \makecell[tl]{-m: \((48,16)\) \\ -n: \((80,48)\)} & \((16,8)\) \\[0.3em]
\(N_7 = (8,24)\) & \makecell[tl]{-o: \((8,24),(8,8),(0,16),(16,16)\) \\ -p: \((8,24),(0,32),(8,40),(16,32)\)} & \makecell[tl]{-o: \((32,64)\) \\ -p: \((32,128)\)} & \((8,24)\) \\[0.3em]
\(N_8 = (-16,32)\) & \makecell[tl]{-q: \((-16,32),(-16,16),(0,16),(0,32)\) \\ -r: \((-16,32),(-32,32),(-32,48),(-16,48)\)} & \makecell[tl]{-q: \((-32,96)\) \\ -r: \((-96,160)\)} & \((-16,32)\) \\
\end{longtable}
\end{document}